\documentclass[11pt]{amsart}

\usepackage{amsmath,amssymb,amsthm,amsfonts,mathrsfs}
\usepackage{bm,cite,array}
\usepackage[normalem]{ulem}
\usepackage{xcolor,graphicx,caption,subfig,pifont,xy}
\usepackage[pagewise]{lineno}
\usepackage{mathtools,enumitem,needspace}
\usepackage[hidelinks]{hyperref}

\makeatletter
\@addtoreset{equation}{section}
\makeatother

\theoremstyle{plain}

\newtheorem*{question}{Question}
\newtheorem{lemma}{Lemma}[section]
\newtheorem{theorem}[lemma]{Theorem}

\newtheorem{definition}[lemma]{Definition}
\newtheorem{corollary}[lemma]{Corollary}
\newtheorem{proposition}[lemma]{Proposition}
\newtheorem{remark}[lemma]{Remark}

\newcommand{\R}{\mathbb R}
\newcommand{\Z}{\mathbb Z}
\newcommand{\T}{\mathbb T}
\newcommand{\av}[1]{\langle #1\rangle}
\newcommand{\norm}[1]{\lVert #1\rVert}
\newcommand{\supp}{\operatorname{supp}}
\newcommand{\osc}{\operatorname{osc}}
\newcommand{\dd}{\,d}
\newcommand{\dx}{\,d\bar x}
\newcommand{\eps}{\varepsilon}

\newcommand{\beq}{\begin{equation}}
\newcommand{\eeq}{\end{equation}}
\newcommand{\beqs}{\begin{eqnarray*}}
\newcommand{\eeqs}{\end{eqnarray*}}
\newcommand{\beqn}{\begin{eqnarray}}
\newcommand{\eeqn}{\end{eqnarray}}
\newcommand{\beqa}{\begin{array}}
\newcommand{\eeqa}{\end{array}}

\begin{document}

\title{A counterexample to integral Aleksandrov-Bakelman-Pucci  estimates for a nonlocal operator}
\renewcommand{\shorttitle}{\protect\resizebox{0.9\textwidth}{!}{A counterexample to integral Aleksandrov-Bakelman-Pucci  estimates for a nonlocal operator}}

\subjclass[2020]{%
	{35B50, 35R09, 35B45, 35R11.}%
}

\author{Yongsheng Jiang}
\address{Yongsheng Jiang, School of Statistics and Mathematics,
	Zhongnan University of Economics and Law,
	Wuhan, Hubei 430073, China}
\email{jiangys@zuel.edu.cn}

\author{Lei Xu}
\address{Lei Xu, School of Statistics and Mathematics,
	Zhongnan University of Economics and Law,
	Wuhan, Hubei 430073, China}
\email{leixumath@zuel.edu.cn}

\begin{abstract}
We construct a counterexample to integral
Aleksandrov-Bakelman-Pucci estimates for nonlocal
inequalities of the form
\begin{equation*}
 L_au(x):=\int_{\mathbb R^n}\delta_yu(x)
 \frac{a(x,y)}{|y|^{n+\sigma}}\,dy
 \ge-f(x)\quad\text{in }B_1,
\end{equation*}
with $u\le0$ in $\mathbb R^n\setminus B_1$,
where $\delta_yu(x)=u(x+y)+u(x-y)-2u(x)$,
and the measurable coefficient satisfies
\begin{equation}\tag{C}
 \lambda\le a(x,y)\le\Lambda,\qquad
 a(x,y)=a(x,-y),\qquad
 x\in\mathbb R^n,\quad y\ne0.
\end{equation}
More precisely, for every $n\ge1$, $0<\sigma<2$,
and $0<\lambda<\Lambda$, there exist a fixed coefficient
$a(x,y)$ satisfying \textup{(C)}, a constant $d>0$, and
nonnegative functions $u_j,f_j\in C_c^\infty(B_1)$
such that $L_au_j\ge-f_j$ in $B_1$,
$0\le f_j\le1$, and
$|\operatorname{supp}f_j|\to0$ as $j\to\infty$,
while $u_j\ge d$ on $B_{1/4}$ for every $j$.
Consequently, this fixed operator admits no
$L^p$ maximum estimate for any $1\le p<\infty$.
\end{abstract}

\maketitle

\baselineskip=16.4pt
\parskip=3pt
\noindent\textbf{Keywords:} Nonlocal elliptic equations;  Aleksandrov-Bakelman-Pucci estimate; maximum principle; perturbation.
\section{Introduction}
The Aleksandrov-Bakelman-Pucci (ABP) estimate bounds the
maximum of an elliptic subsolution in terms of an integral
norm of the inhomogeneous term and is a basic tool in
elliptic regularity theory~\cite{CC95}. For uniformly elliptic second-order equations in
nondivergence form, the classical ABP estimate holds
with merely measurable leading coefficients \cite[Chapter~9]{GT01}.
Let $n\ge1$, and denote by $B_r$ the ball of radius $r$
centered at the origin. Suppose that $A(x)=A(x)^T$ is
measurable and satisfies $\lambda I\le A(x)\le\Lambda I$,
where $0<\lambda<\Lambda$. If
$u\in C^2(B_1)\cap C(\overline{B_1})$ satisfies
$\operatorname{tr}(A(x)D^2u)\ge-f$ in $B_1$ and
$u\le0$ on $\partial B_1$, with $0\le f\in L^n(B_1)$,
then
\begin{equation}\label{eq:abpest}
 \sup_{B_1}u^+\le\frac{C_n}{\lambda}\|f\|_{L^n(B_1)},
\end{equation}
where $u^+=\max\{u,0\}$. In particular,
$0\le f\le\mathbf1_E$ implies
$\sup_{B_1}u^+\le(C_n/\lambda)|E|^{1/n}$, where $|E|$ denotes the Lebesgue measure and $\mathbf1_E$
the indicator function of $E$.

We examine whether integral ABP estimates analogous to
\eqref{eq:abpest} hold for nonlocal linear equations with
even uniformly elliptic kernels. Fix $0<\sigma<2$, and consider 
\begin{equation}\label{operator}
 L_au(x):=\int_{\R^n}\delta_yu(x)
 \frac{a(x,y)}{|y|^{n+\sigma}}\dd y
 \ge-f(x)\quad\text{in }B_1,
\end{equation}
with $f\ge0$ and $u\le0$ in $\R^n\setminus B_1$,
where $\delta_yu(x)=u(x+y)+u(x-y)-2u(x)$, and the
measurable coefficient satisfies
\begin{equation}\label{ellipticity}
 \lambda\le a(x,y)\le\Lambda,\qquad
 a(x,y)=a(x,-y),\qquad x\in\R^n,\quad y\ne0.
\end{equation}

Guillen-Schwab~\cite[Question~1.1]{GS12} ask under what
assumptions on the operator $L$ the solutions of $Lv_j=f_j$
in $B_1$, with $v_j=0$ in $\R^n\setminus B_1$ and
$0\le f_j\le1$, satisfy
$\|v_j\|_{L^\infty(B_1)}\to0$ whenever
$|\{f_j>0\}|\to0$.
The Nonlocal PDE Wiki~\cite{WikiABP} formulates a related
question for inequalities with even uniformly elliptic
kernels. In the setting of
\eqref{operator}--\eqref{ellipticity}, this question
takes the following form.\footnote{The webpage displays
a positive indicator on the right-hand side. With its
operator convention and nonpositive exterior values,
the maximum principle would then give $u\le0$.
We therefore use $-\mathbf1_{E_j}$ below.
By evenness, the first-difference operator on the webpage
agrees with $L_a$ for
$K(x,y)=2a(x,y)|y|^{-n-\sigma}$, with the integral
interpreted in the symmetric principal-value sense
when necessary.}

\begin{question}
For a fixed coefficient $a$ satisfying \eqref{ellipticity},
let $E_j\subset B_1$ be open and let
$u_j\in C_b^2(\R^n)$ satisfy
\begin{equation}\label{small-support-problem}
 L_au_j\ge-\mathbf1_{E_j}\quad\text{in }B_1,
 \qquad u_j\le0\quad\text{in }\R^n\setminus B_1.
\end{equation}
Does $\lim_{j\to\infty}|E_j|=0$ imply
$\lim_{j\to\infty}\sup_{B_1}u_j^+=0$?
\end{question}

An integral ABP estimate of the form
$\sup_{B_1}u^+\le C\|f\|_{L^p(B_1)}$ for some
$1\le p<\infty$, 
would answer this question affirmatively. Indeed,
taking $f=\mathbf1_{E_j}$ gives
$
 \sup_{B_1}u_j^+
 \le C\|\mathbf1_{E_j}\|_{L^p(B_1)}
 =C|E_j|^{1/p}\rightarrow0
$
whenever $|E_j|\to0$.
The question about \eqref{small-support-problem} therefore asks only for a qualitative
consequence of integral ABP estimates, without specifying
an exponent.

For general uniformly elliptic nonlocal equations,
Caffarelli--Silvestre~\cite{CS09} prove an ABP-type estimate
involving a covering of the contact set and local suprema
of the inhomogeneous term. This estimate yields Harnack
inequalities and regularity results, but does
not by itself imply that $\sup_{B_1}u_j^+\rightarrow0$
whenever $|E_j|\rightarrow0$.

Integral ABP estimates are available under additional
assumptions on the kernel.
Guillen--Schwab~\cite{GS12} obtain such an
estimate which, for $n\ge2$, applies in particular to
coefficients of the form
\begin{equation}\label{quadratic-kernel}
 a(x,y)=\frac{y^TA(x)y}{|y|^2},\qquad
 A(x)=A(x)^T,\qquad
 \lambda I\le A(x)\le\Lambda I.
\end{equation}
For $u\in C_b^2(\R^n)$ satisfying \eqref{operator},
with $0\le f\in C(\overline{B_1})$ and
$u\le0$ in $\R^n\setminus B_1$, Theorem 1.3 in \cite{GS12},
applied to $-u$, gives
$
 \sup_{B_1}u^+
 \le C(n,\sigma,\lambda,\Lambda) \|f\|_{L^\infty(B_1)}^{1-\sigma/2}
       \|f\|_{L^n(B_1)}^{\sigma/2}.$
For the same kernel class \eqref{quadratic-kernel},
Kitano~\cite{Kitano20, Kitano22} obtains an estimate
involving only the $L^n$ norm.

Our main theorem gives a negative answer to the question
under \eqref{ellipticity}. 

\begin{theorem}\label{main}
Let $n\ge1$, $0<\sigma<2$, and $0<\lambda<\Lambda$. There exist a coefficient
$
 a\in C^\infty\bigl(\R^n\times(\R^n\setminus\{0\})\bigr)
$
satisfying \eqref{ellipticity}, a constant $d=d(n,\sigma,\lambda,\Lambda)>0$, open sets $E_j\subset B_1$, and functions $u_j,f_j\in C_c^\infty(B_1)$ such that $u_j\ge0$, $0\le f_j\le1$, $\supp f_j\subset E_j$ with $\lim_{j\to\infty}|E_j|=0$, and
$
 L_au_j\ge-f_j\text{ in } B_1$, while  $u_j\ge d$ in $B_{1/4}$.

\end{theorem}

Since the coefficient $a$ is fixed, the following
consequence holds even when the constant in the estimate
is allowed to depend on $a$.

\begin{corollary}\label{noestimate}
For the operator in Theorem~\ref{main} and every $1\le p<\infty$, there is no constant $C_a$ such that
$
 \sup_{B_1}u^+\le C_a\|f\|_{L^p(B_1)}
$
holds for all $u,f\in C_c^\infty(B_1)$ with $0\le f\le1$ and $L_au\ge-f$ in $B_1$.  
\end{corollary}

The coefficient constructed in the proof of
Theorem~\ref{main} is radial in $y$, with a nontrivial
dependence on $|y|$; see Remark~\ref{radialdependence} below.
It therefore lies outside the class \eqref{quadratic-kernel},
whose coefficients are homogeneous of degree zero in $y$. The following corollary addresses the question posed by
Guillen--Schwab~\cite[Question~1.1]{GS12}.
It shows that evenness and uniform ellipticity alone
do not ensure uniform convergence of the Dirichlet
solutions to zero as the supports of the data shrink
in measure.

\begin{corollary}\label{equationcounterexample}
For the coefficient and data in Theorem~\ref{main}, there exist bounded continuous functions $w_j:\R^n\to\R$ satisfying
\begin{equation}\label{dirichletcounterexample}
 L_aw_j=-f_j\quad\text{in }B_1,\qquad
 w_j=0\quad\text{in }\R^n\setminus B_1,
\end{equation}
in the viscosity sense. Moreover,   $w_j\ge d$ on $B_{1/4}$, where $d>0$ is the constant in Theorem ~\ref{main}. Consequently, $\widetilde w_j=-w_j$ solves $L_a\widetilde w_j=f_j$ in $B_1$ in the viscosity sense with zero exterior values, while
$
 \lim_{j\to\infty}|\{f_j>0\}|=0$,
 and $\|\widetilde w_j\|_{L^\infty(B_1)}\ge d$ for every $j$.
\end{corollary}

The construction is based on the adjoint equation 
for a periodic Poisson problem. Section~\ref{sec:poisson} shows that concentration of the normalized positive adjoint solution yields a compactly supported subsolution with a uniform positive lower bound. Section~\ref{sec:concentration} obtains the required concentration by introducing rapidly oscillating coefficients at successively smaller jump lengths. Section~\ref{sec:fixed} proves the main results. The periodic solvability theorem is proved in Appendix~\ref{sec:linear}.

\section{Preliminaries and periodic problems}\label{sec:poisson}

We first collect the notation and solvability results for the  adjoint and Poisson
equations. We then prove an oscillation estimate and a localization result whose constants are independent of coefficient derivatives. Throughout this section,  $L_a$ is defined by \eqref{operator} and $a$ satisfies \eqref{ellipticity}. We denote by $C_b^2(\R^n)$ the space of $C^2$ functions
on $\R^n$ that are bounded together with their first
and second derivatives. Taylor's formula gives
\begin{equation}\label{incrementbound}
 |\delta_y\varphi(x)|\le\min\{\|D^2\varphi\|_\infty|y|^2,\,4\|\varphi\|_\infty\},
 \qquad \varphi\in C_b^2(\R^n),
\end{equation}
where the Hessian is measured in the matrix operator norm. 

Fix $\ell>0$ and let $\T=\R^n/\ell\Z^n$, represented by $Q_\ell=[-\ell/2,\ell/2]^n$. Functions on $\T$ are identified with $\ell$- periodic functions on $\R^n$. We use normalized integration and norms:
\begin{equation*}
 \av h=\ell^{-n}\int_{Q_\ell}h(x)\dd x=\int_\T h\dx,
 \qquad \|h\|_p=\|h\|_{L^p(\T)}.
\end{equation*}
All function spaces below are consising of real functions, and $(u,v)_2=\av{uv}$ denotes the real $L^2$ inner product. Complex exponentials are used for Fourier calculations only, with distributions and operators extended complex-linearly. Let
\begin{equation}\label{fouriernotation}
 e_k(x)=e^{2\pi i k\cdot x/\ell},\qquad
 \widehat v(k)=\av{v\overline{e_k}},\qquad k\in\Z^n,
\end{equation}
and
\begin{equation}\label{sobolevnotation}
 H^s(\T)=\left\{v\in\mathcal D'(\T):
 \|v\|_{H^s}^2:=\sum_{k\in\Z^n}(1+|k|^2)^s|\widehat v(k)|^2<\infty\right\},
 \qquad s\in\R.
\end{equation}
 For a coefficient $a$ that is $\ell$-periodic in $x$, the adjoint $L_a^*$ is defined by
\begin{equation}\label{adjointdefinition}
 \langle L_a^*\phi,\varphi\rangle_\T=\av{\phi L_a\varphi},
 \qquad \phi\in L^1(\T),\quad\varphi\in C^\infty(\T).
\end{equation}
\begin{definition}\label{adjsol}
For a coefficient $a$ that is $\ell$-periodic in $x$, let $L_a^*$ be the distributional adjoint of $L_a$ defined by \eqref{adjointdefinition}. Suppose that $\phi\in  L^2(\T) $ satisfies $\phi>0$, $L_a^*\phi=0$,
and $\av{\phi}=1$,  that is,
 \begin{equation}\label{adjsolution}
 \langle L_a^*\phi,\varphi\rangle_\T=\av{\phi L_a\varphi}=0, \text{ for  all } \varphi\in C^\infty(\T), \av{\phi}=1.
\end{equation}
We call $\phi$
the normalized positive adjoint solution of $L_a$.
\end{definition}

More generally, for an operator $L$ as in Theorem~\ref{periodictheorem}
below, the function $\phi$ given by
Theorem~\ref{periodictheorem}\,\textup{(ii)} below is called the normalized
positive adjoint solution of~$L$.

Fix $a_*\in(\lambda,\Lambda)$. Each finite stage of the construction will belong to the following coefficient class.

\begin{definition}\label{finiteclass}
A coefficient \(a\) is called an auxiliary coefficient if $a(x,y)$ satisfies \eqref{ellipticity}, is smooth for $y\ne0$, $\ell$-periodic in $x$, and radial in $y$, and, for some $0<\rho<R$ and smooth positive periodic $c(x)$, satisfies
\begin{equation*}
 a(x,y)=a_*c(x) \text{ for }|y|<\rho,\qquad
 a(x,y)=a_*\text{ for }|y|>R.
\end{equation*}
The parameters $a_*,\ell$ are fixed; $\rho,R,c(x)$ may vary.
\end{definition}

Fix a radial $\chi\in C_c^\infty(B_1)$ with $0\le\chi\le1$ and $\chi=1$ on $B_{1/2}$. Define $\chi_r(y)=\chi(y/r)$ and
\begin{equation}\label{Ardefinition}
 A_rv(x)=a_*\int_{\R^n}\delta_yv(x)\frac{\chi_r(y)}{|y|^{n+\sigma}}\dd y.
\end{equation}
This operator is self-adjoint and nonpositive in $L^2(\T)$, with domain $H^\sigma(\T)$, and
\begin{equation}\label{ellest}
 \|v\|_{H^{s+\sigma}}\le C_{r,s}
 \bigl(\|A_rv\|_{H^s}+\|v\|_{H^s}\bigr),\qquad s\in\R;
\end{equation}
see Lemma~\ref{fourier} below. For an auxiliary coefficient and any $r>0$, the periodized tail kernel is
\begin{equation}\label{periodkernel}
 q(x,z)=2\ell^n\sum_{k\in\Z^n}
 \frac{(1-\chi_r(z-x+\ell k))a(x,z-x+\ell k)}
 {|z-x+\ell k|^{n+\sigma}}.
\end{equation}
 Lemma~\ref{periodization}  below gives
\begin{equation}\label{qassumption}
 q\in C^\infty(\T\times\T),\qquad 0<q_-\le q(x,z)\le q_+.
\end{equation}
Define
\begin{equation}\label{Qdefinition}
 Qv(x)=\ell^{-n}\int_{Q_\ell}q(x,z)v(z)\dd z,\qquad b(x)=\ell^{-n}\int_{Q_\ell}q(x,z)\dd z.
\end{equation}
Then $Q1=b$, $q_-\le b\le q_+$. Moreover, \eqref{periodkernel} and \eqref{Qdefinition}   show that, for every $v\in C^\infty(\T)$(see Lemma A.2 below), 
\begin{equation}\label{tailrepresentation}
 \int_{\R^n}\delta_yv(x)\frac{(1-\chi_r(y))a(x,y)}{|y|^{n+\sigma}}\dd y
 =Qv(x)-b(x)v(x).
\end{equation}
For $r\le\rho$ in Definition~\ref{finiteclass}, it follows from  Lemma~\ref{periodization}  below  that 
\begin{equation}\label{Lmodel}
 L_a=c(x)A_r+Q-b.
\end{equation}
\Needspace{7\baselineskip}
\begin{theorem}
\label{periodicsolvability}\label{periodictheorem}
Let $r>0$, $c\in C^\infty(\T)$ with $\min_\T c>0$,
and let $q$ satisfy \eqref{qassumption}.
Define $L=cA_r+Q-b$, where $A_r$ is given by
\eqref{Ardefinition} and $Q,b$ by \eqref{Qdefinition}.
Here $L^*$ denotes the real adjoint, defined
as in \eqref{adjointdefinition} with $L_a$ replaced by $L$.
Then the following assertions hold.
\begin{enumerate}[
 label=\textup{(\roman*)},
 leftmargin=*,
 itemsep=0.2em
]
\item\label{periodicfredholm}
$L:H^\sigma(\T)\to L^2(\T)$ is Fredholm of index zero,
with $\ker L=\operatorname{span}\{1\}$.
If $v\in H^\sigma(\T)$ and $Lv\in C^\infty(\T)$,
then $v\in C^\infty(\T)$.

\item\label{periodicadjoint}
There is a unique $\phi\in C^\infty(\T)$ satisfying
$L^*\phi=0$, $\phi>0$, and $\av{\phi}=1$.
The  kernel of $L^*$ in $L^2(\T)$
is $\operatorname{span}\{\phi\}$.

\item\label{periodicpoisson}
For real-valued $f\in C^\infty(\T)$, the equation
$Lv=f$ has a real-valued smooth solution if and only if
$\av{f\phi}=0$.
The solution is unique up to an additive constant;
either $\av{v\phi}=0$ or $\min_\T v=0$
determines a unique solution.
\end{enumerate}

\end{theorem}

The proof is given in Appendix~\ref{sec:linear}.
 Given $h\in C^\infty(\T)$ with $0\le h\le1$, for $ \phi\in L^1(\T)$,  define
\begin{equation}\label{chdefinition}
 c_h=\av{h\phi}.
\end{equation}
Then  $\av{(c_h-h)\phi}=0$, and
Theorem~\ref{periodictheorem}\,\textup{(iii)}
gives a unique smooth solution of
\begin{equation}\label{cellpoisson}
 L_av=c_h-h,\qquad \min_\T v=0.
\end{equation}
 Fix a smooth $\eta:\R\to[0,1]$ with $\eta=0$ on $(-\infty,1]$ and $\eta=1$ on $[2,\infty)$, and define
\begin{equation}\label{extractiondata}
 h=\eta(\phi),\qquad F=\{x\in\T:\phi(x)>1/2\}.
\end{equation}
Define
\begin{equation}\label{Jdefinition}
 J(\phi)=\av{\sqrt{\phi}}.
\end{equation}
The next lemma shows that if $J(\phi)$ is small, then the function $h$ in \eqref{extractiondata} is supported in a set of small measure and $c_h=\av{h\phi}$ is close to $1$.
\begin{lemma}\label{extraction} Let $h$ and $F$ be defined by \eqref{extractiondata}.
If $\phi>0$ is smooth with $\av{\phi}=1$, then $\supp_\T h\subset F$, $|F|/\ell^n\le\sqrt2 J(\phi)$, and $c_h=\av{h\phi}\ge1-\sqrt2 J(\phi)$.
\end{lemma}
\begin{proof}
The support inclusion follows from $\supp_\T h\subset\{\phi\ge1\}\subset F$. Since $\sqrt{2\phi}>1$ on $F$, we have 
\begin{equation*}
 \frac{|F|}{\ell^n}=\av{\mathbf1_F}\le \frac{1}{\ell^n}\int_{F}\sqrt{2\phi(x)}dx\le\sqrt2\av{\sqrt{\phi}},
\end{equation*}
and
\begin{equation*}
 1-c_h=\av{(1-h)\phi}\le\av{\mathbf1_{\{\phi<2\}}\phi} \le\av{\sqrt{2}\sqrt{\phi}}
 \le\sqrt2 J(\phi).
\end{equation*}
\end{proof}

We next estimate the oscillation of a periodic solution.

\begin{lemma}\label{oscillation}
Suppose that $a$ satisfies \eqref{ellipticity} and is $\ell$-periodic in $x$. If $v\in C^\infty(\T)$ satisfies $L_av=f$, then
\begin{equation*}
 \osc v\le \frac{n^{(n+\sigma)/2}}{2\lambda}\ell^\sigma\osc f.
\end{equation*}
where $\osc v=\max_\T v-\min_\T v$, and $\osc f$ is defined in the same way.
\end{lemma}
\begin{proof}
Denote by $x_+,x_-\in Q_\ell$ a maximum and a minimum point of $v$, respectively.  At $x_+$, the integrand of $L_a v(x_+)$ is nonpositive. Since $|z-x_+|\le\sqrt n\ell$ for $z\in Q_\ell$, restricting the integral to this cell $Q_\ell$ gives
\begin{equation*}
 \begin{aligned}
 f(x_+)&=2\int_{\R^n}(v(x_++y)-v(x_+))
       \frac{a(x_+,y)}{|y|^{n+\sigma}}\dd y\\
 &\le\frac{2\lambda}{(\sqrt n\ell)^{n+\sigma}}
       \int_{Q_\ell}(v(z)-v(x_+))\dd z
 =\frac{2\lambda}{n^{(n+\sigma)/2}}\ell^{-\sigma}(\av v-v(x_+)).
 \end{aligned}
\end{equation*}
At the minimum point, the same argument gives
\begin{equation*}
 f(x_-)\ge \frac{2\lambda}{n^{(n+\sigma)/2}}\ell^{-\sigma}(\av v-v(x_-)).
\end{equation*}
Subtracting the two inequalities yields
$
 \osc v\le \frac{n^{(n+\sigma)/2}}{2\lambda}\ell^\sigma(f(x_-)-f(x_+))
 \le \frac{n^{(n+\sigma)/2}}{2\lambda}\ell^\sigma\osc f$.
\end{proof}

Fix $\theta\in C^\infty(\R^n)$ with $0\le\theta\le1$ such that 
\begin{equation}\label{eq:cutoff}
 \theta(x)=\begin{cases}
  \displaystyle 0,& |x|<1/4,\\
  1,& |x|>3/4.
 \end{cases}
\end{equation}
 By \eqref{incrementbound}, there is $D\ge1$, depending only on $n,\sigma,\Lambda,\theta$, such that $\|L_a\theta\|_\infty\le D$ for every coefficient $a$ satisfying \eqref{ellipticity}. Fix
\begin{equation}\label{dchoice}
 d=\frac1{16(D+1)},
\end{equation}
and choose the period $\ell$ once and for all so that
\begin{equation}\label{ellchoice}
\frac{n^{(n+\sigma)/2}}{2\lambda}\ell^\sigma\le d.
\end{equation}
Next, we choose $H\in C^\infty(\R)$ convex with $H=0$ on $(-\infty,0]$, $0\le H'\le1$, and
\begin{equation}\label{Hbounds}
 s^+-d\le H(s)\le s^+,\qquad s\in\R. 
\end{equation}
For example, take $H(s)=\int_0^s\gamma(t/d)\dd t$, where $\gamma$ is smooth, nondecreasing, $\gamma(t)=0$  for $t\leq0$, and $\gamma(t)=1$ for $t\geq1$. Fix $\zeta\in C_c^\infty(B_1)$ with $0\le\zeta\le1$ and $\zeta=1$ on $B_{3/4}$.

\begin{proposition}\label{localization}
Let $a$ be an auxiliary coefficient, and let
$h,c_h,v$ satisfy \eqref{chdefinition}-\eqref{cellpoisson},
with $0\le h\le1$.
With $\theta$ fixed in \eqref{eq:cutoff}, the following assertions hold.
\begin{enumerate}[label=\textup{(\roman*)},leftmargin=*,itemsep=0.2em]
\item $0\le v\le d$, and $z=v+2d-4d\theta$ satisfies $z\ge2d$ on $B_{1/4}$ and $z\le-d$ outside $B_{3/4}$.
\item If $c_h\ge3/4$ and $\widetilde a$ satisfies \eqref{ellipticity} and
\begin{equation}\label{errorbudget}
 \|(L_{\widetilde a}-L_a)z\|_\infty\le\tfrac14,
\end{equation}
then $u=H(z)\in C_c^\infty(B_{3/4})$ and $f=\zeta h\in C_c^\infty(B_1)$ satisfy $u\ge0$, $0\le f\le1$, $\supp f\subset\supp\zeta\cap\supp h$, and
\begin{equation*}
 L_{\widetilde a}u\ge-f\quad\text{in }B_1,\qquad
 u\ge d\quad\text{on }B_{1/4}.
\end{equation*}
\end{enumerate}
\end{proposition}
\begin{proof}By 
Lemma~\ref{oscillation}, \eqref{ellchoice} and $\min_\T v=0$, we have 
\begin{equation*}
 0\le v\le\osc v\le \frac{n^{(n+\sigma)/2}}{2\lambda}\ell^\sigma\osc h\le d.
\end{equation*}
The definition of $\theta$ then gives $z\ge2d$ on $B_{1/4}$ and $z\le-d$ outside $B_{3/4}$. This proves \textup{(i)}.

Since $z\le-d$ on $\partial B_{3/4}$ and outside $B_{3/4}$, continuity implies that $H(z)$ vanishes in a neighborhood of  $\partial B_{3/4}$ and outside $B_{3/4}$. Hence $u=H(z)\in C_c^\infty(B_{3/4})$, and \eqref{Hbounds} gives $u\ge z-d\ge d$ on $B_{1/4}$. The support and bounds for $f$ follow directly from $f=\zeta h$.

To prove the subsolution inequality, we combine \eqref{cellpoisson}, \eqref{dchoice}  and \eqref{errorbudget} with $c_h\geq 3/4$ and $\|L_a\theta\|_\infty$, we get
\begin{equation}\label{eq:Lsup}
 L_{\widetilde a}z\ge c_h-h-4dL_a\theta-\tfrac14
 \ge-h+\tfrac12-4dD\ge-h+\tfrac14.
\end{equation}
By the convexity of $H$,
\begin{equation}\label{eq:convH}
 \delta_yH(z)(x)\ge H'(z(x))\delta_yz(x).
\end{equation}
Since $u,z\in C_b^2(\R^n)$, the corresponding integrals
are absolutely convergent by \eqref{incrementbound}.
Integrating against the positive kernel of
$L_{\widetilde a}$, it follows from \eqref{eq:Lsup}, \eqref{eq:convH}  and $H'(z)\geq0$ that
$
 L_{\widetilde a}u\ge H'(z)L_{\widetilde a}z\ge -H'(z)h+\tfrac14H'(z)
 \ge -H'(z)h
 \ge -h$ in $B_1$,
where the last two inequalities follow from
$h\ge0$ and $0\le H'\le1$. Since $\zeta=1$ on $B_{3/4}$, we have $f=\zeta h=h$ in $B_{3/4}$, and hence $L_{\widetilde a}u\ge-f$ in $B_{3/4}$.
For $x\in B_1\setminus B_{3/4}$, we have $u(x)=0$ and
$u\ge0$ in $\R^n$. The positivity of the kernel therefore
gives
$
 L_{\widetilde a}u(x)\ge0\ge-f(x)$ for $x\in  B_1\setminus B_{3/4}.
$
Thus $L_{\widetilde a}u\ge-f$ in $B_1$.
\end{proof}

\section{Concentration of adjoint solutions}\label{sec:concentration}

Choose a positive, nonconstant, smooth $1$-periodic function $g$ such that
\begin{equation}\label{gchoice}
 \int_0^1g(t)\dd t=1,\qquad \lambda\le \frac{a_*}{g(t)}\le\Lambda.
\end{equation}
For example, one may take $g(t)=1+\eps\cos(2\pi t)$ with $\eps>0$ sufficiently small. For each positive integer $N$, define
\begin{equation}\label{gNdefinition}
 g_N(x)=g(Nx_1/\ell), \ x=(x_1,x_2,\cdots, x_n)\in\mathbb{R}^n.
\end{equation}
Since $g$ is nonconstant and has mean one, the strict form of Jensen's inequality gives
\begin{equation}\label{rhochoice}
 \rho_g:=\int_0^1\sqrt{g(t)}\dd t<1.
\end{equation}
Fix $\tau\in(\rho_g,1)$. Both $g$ and $\tau$ remain fixed throughout this section.  

\begin{proposition}\label{onestep}
Let $a$ be an auxiliary coefficient as in Definition~\ref{finiteclass}, let  $g_N$  and $\rho_g$ be given by \eqref{gNdefinition}, and \eqref{rhochoice}, respectively, 
 and let $\phi$ be the normalized positive adjoint solution of $L_a$( see Definition~\ref{adjsol}). Define $\chi_r$ as in \eqref{Ardefinition}. Given any $R_0>0$ and integer $N_0\ge0$, there exist $0<r<R_0$ and an integer $N>N_0$ such that the coefficient
\begin{equation}\label{newcoefficient}
 a_{r,N}(x,y)=\chi_r(y)\frac{a_*}{g_N(x)}+(1-\chi_r(y))a(x,y)
\end{equation}
is an auxiliary coefficient, and the normalized positive adjoint solution $\phi_{r,N}$ of $L_{a_{r,N}}$ satisfies
\begin{equation}\label{onestepcontraction}
 J(\phi_{r,N})\le\tau J(\phi),
\end{equation}
where $\tau\in(\rho_g,1)$.
\end{proposition}

We first establish three lemmas used in the proof.

\begin{lemma}\label{smalljump}
Let $a,\widetilde a$ satisfy \eqref{ellipticity}, and suppose that $a=\widetilde a$ whenever $|y|\ge r$. For every $\varphi\in C_b^2(\R^n)$,
\begin{equation}\label{smalljumpestimate}
 \norm{(L_a-L_{\widetilde a})\varphi}_\infty
 \le \frac{2\Lambda|\mathbb S^{n-1}|}{2-\sigma}
 r^{2-\sigma}\norm{D^2\varphi}_\infty.
\end{equation}
For later use, denote
$C_*:=2\Lambda|\mathbb S^{n-1}|/(2-\sigma)$.
\end{lemma}
\begin{proof}
Since $a=\widetilde a$ for $|y|\ge r$, \eqref{incrementbound} gives
\begin{equation*}
 \|(L_a-L_{\widetilde a})\varphi\|_\infty
 \le2\Lambda\|D^2\varphi\|_\infty
       \int_{|y|<r}|y|^{2-n-\sigma}\dd y
 =\frac{2\Lambda|\mathbb S^{n-1}|}{2-\sigma}r^{2-\sigma}\|D^2\varphi\|_\infty.
\end{equation*}
\end{proof}

\begin{lemma}\label{averaging}
Let $G$ be a continuous $1$-periodic function in $\mathbb{R}$. For every $\Phi\in L^1(\T)$, as $N\rightarrow\infty$ through the positive integers,
\begin{equation}\label{averagelimit}
 \lim_{N\to\infty}\av{\Phi(x)G(Nx_1/\ell)}
 =\av{\Phi}\int_0^1G(t)\dd t.
\end{equation}
where $x_1$ denotes the first component of $x$.
\end{lemma}
\begin{proof}
Let $P(x)$ be a trigonometric polynomial,
$P(x)=\sum_{k\in\mathcal F}\widehat P(k)e_k(x)$,
where $\mathcal F\subset\Z^n$ is finite and $e_k$
is defined in \eqref{fouriernotation}.
Choose $N>\max_{k\in\mathcal F}|k_1|$.
 For each $k\in\mathcal F$, partition $[0,\ell]$ into
$N$ equal intervals and set $x_1=\ell(j+t)/N$
on the $j$th interval. Since $G(j+t)=G(t)$, we obtain
\begin{equation*}
 \frac1\ell\int_0^\ell
 e^{2\pi i k_1x_1/\ell}G(Nx_1/\ell)\dd x_1=
 \frac1N\sum_{j=0}^{N-1}
 e^{2\pi i k_1j/N}
 \int_0^1e^{2\pi i k_1t/N}G(t)\dd t.
\end{equation*}
  If $0<|k_1|<N$, then $e^{2\pi i k_1/N}\ne1$,
whereas $e^{2\pi i k_1}=1$ since $k_1\in\Z$.
Summing the geometric series, we obtain
\begin{equation*}
 \frac1N\sum_{j=0}^{N-1}e^{2\pi i k_1j/N}
 =
 \begin{cases}
  1,& k_1=0,\\
  0,& k_1\ne0.
 \end{cases}
\end{equation*}
For each $m=2,\ldots,n$, 
\begin{equation*}
 \frac1\ell\int_0^\ell
 e^{2\pi i k_mx_m/\ell}\dd x_m
 =
 \begin{cases}
  1,& k_m=0,\\
  0,& k_m\ne0.
 \end{cases}
\end{equation*}
Since $G(Nx_1/\ell)$ depends only on $x_1$,  the preceding calculation implies that,
for $N>|k_1|$,
\begin{equation*}
 \av{e_k(x)G(Nx_1/\ell)}
 =
 \begin{cases}
  \displaystyle\int_0^1G(t)\dd t,& k=0,\\
  0,& k\ne0.
 \end{cases}
\end{equation*}
It follows from $P(x)=\sum_{k\in\mathcal F}\widehat P(k)e_k(x)$ that 
\begin{equation*}
 \av{P(x)G(Nx_1/\ell)}
 =\widehat P(0)\int_0^1G(t)\dd t
 =\av{P}\int_0^1G(t)\dd t.
\end{equation*}
Now let $\Phi\in L^1(\T)$ and $\varepsilon>0$.
Choose a triagonometric polynomial $P$ such that $\|\Phi-P\|_1<\varepsilon$; this is possible because trigonometric polynomials are dense in $L^1(\T)$.
For all  $N>N_P:=\max_{k\in\mathcal F}|k_1|$,  the identity above yields
\begin{equation*}
 \begin{aligned}
 \left|
 \av{\Phi(x)G(Nx_1/\ell)}
 -\av{\Phi}\int_0^1G(t)\dd t
 \right|&=
 \left|
 \av{(\Phi-P)(x)G(Nx_1/\ell)}
 -\av{\Phi-P}\int_0^1G(t)\dd t
 \right|\\
 &\qquad\le
 2\|G\|_{L^\infty(0,1)}\|\Phi-P\|_1\le
 2\varepsilon\|G\|_{L^\infty(0,1)}.
 \end{aligned}
\end{equation*}
Taking the limit superior as $N\to\infty$ and then letting
$\varepsilon\to0^+$ proves \eqref{averagelimit}.
\end{proof}

\begin{lemma}\label{highfreq}
Fix $r>0$, and let $q$ satisfy \eqref{qassumption}. Let $A_r$ be defined by \eqref{Ardefinition}, and define $Q,b$ by \eqref{Qdefinition} for this kernel $q$. Let $g$ be a positive smooth $1$-periodic function with $\int_0^1g(t)\dd t=1$, and define $g_N$ by \eqref{gNdefinition}. Set
\begin{equation}\label{highfreqoperators}
 L_N=g_N^{-1}A_r+Q-b,\qquad L_0=A_r+Q-b.
\end{equation}
Let $\phi_N$ and $\phi_0$ be the normalized positive adjoint solutions of $L_N$ and $L_0$, respectively, given by Theorem~\ref{periodictheorem}. Then
\begin{equation}\label{strongapprox}
 \lim_{N\to\infty}\norm{\phi_N-g_N\phi_0}_2=0.
\end{equation}
\end{lemma}
\begin{proof}
The operators in \eqref{highfreqoperators} satisfy the hypotheses of Theorem~\ref{periodictheorem} with $c=g_N^{-1}$ and $c=1$, respectively. We keep $r,q,g,\ell$ fixed throughout this proof. All constants below may depend on these data, but are independent of $N$.

\smallskip\noindent
\emph{Step 1. Uniform estimates for $w_N=\phi_N/g_N$.}
 By \eqref{Qdefinition}, the adjoint of $Q$ with respect to normalized integration is
$
 Q^*\varphi(x)=\ell^{-n}\int_{Q_\ell}q(z,x)\varphi(z)\dd z$.
Since $A_r$ is self-adjoint, \eqref{highfreqoperators} and $L_N^*\phi_N=0$ give
\begin{equation}\label{wNeq}
 (-A_r+bg_N)w_N=Q^*\phi_N.
\end{equation}
By \eqref{qassumption} and $\av{\phi_N}=1$, we have $q_-\le Q^*\phi_N\le q_+$. Let $x_N$ be a maximum point of $w_N$. Since $w_N>0$, $b\ge q_-$, and $-A_rw_N(x_N)\ge0$, \eqref{wNeq} yields
\begin{equation*}
 q_-\min_{[0,1]}g\,\|w_N\|_\infty
 \le b(x_N)g_N(x_N)w_N(x_N)
 \le Q^*\phi_N(x_N)\le q_+.
\end{equation*}
It follows that $\|w_N\|_\infty\le C$ and $\|\phi_N\|_\infty\le C$. Returning to \eqref{wNeq}, we obtain
\begin{equation*}
 \|A_rw_N\|_\infty
 \le q_+\|g\|_\infty\|w_N\|_\infty+q_+\le C.
\end{equation*}
Applying \eqref{ellest} with $s=0$ gives
\begin{equation}\label{wcompact}
 \|w_N\|_{H^\sigma}
 \le C_r\bigl(\|A_rw_N\|_2+\|w_N\|_2\bigr)\le C.
\end{equation}
Since $\sigma>0$, the compact embedding $H^\sigma(\T)\subset  L^2(\T)$ shows that $\{w_N\}$ is relatively compact in $L^2(\T)$.

\smallskip\noindent
\emph{Step 2. The limit equation and its normalization.}
By \eqref{wcompact}, any subsequence has a further subsequence, indexed by $\{N_m\}$, and a function $w\in H^\sigma(\T)$ such that
$w_{N_m}\rightharpoonup w$ weakly in $H^\sigma(\T)$
and
\begin{equation}\label{wstrongsubsequence}
 \lim_{m\to\infty}\|w_{N_m}-w\|_2=0.
\end{equation}
We first identify the weak limit of $\phi_{N_m}=g_{N_m}w_{N_m}$. For every $\Phi\in L^2(\T)$,
\begin{equation*}
 \left|\av{\Phi(\phi_{N_m}-w)}\right|
 \le\|g\|_\infty\|\Phi\|_2\|w_{N_m}-w\|_2
       +\left|\av{\Phi w(g_{N_m}-1)}\right|.
\end{equation*}
The first term tends to zero by \eqref{wstrongsubsequence}. Since $\Phi w\in L^1(\T)$ and $\int_0^1g=1$, the second tends to zero by \eqref{averagelimit}. Therefore
\begin{equation}\label{phiweaksubsequence}
 \phi_{N_m}\rightharpoonup w\quad\text{weakly in }L^2(\T).
\end{equation}
Taking $\Phi=1$ in \eqref{phiweaksubsequence} and using $\av{\phi_{N_m}}=1$ gives $\av{w}=1$.

For every $\varphi\in C^\infty(\T)$, the adjoint equation and \eqref{highfreqoperators} give
\begin{equation*}
 \av{w_{N_m}A_r\varphi}
 +\av{\phi_{N_m}(Q\varphi-b\varphi)}=0.
\end{equation*}
Since $A_r\varphi,Q\varphi-b\varphi\in L^2(\T)$ are independent of $m$, \eqref{wstrongsubsequence} and \eqref{phiweaksubsequence} imply
\begin{equation*}
 \av{wL_0\varphi}
 =\lim_{m\to\infty}
   \bigl(\av{w_{N_m}A_r\varphi}
         +\av{\phi_{N_m}(Q\varphi-b\varphi)}\bigr)=0.
\end{equation*}
Thus $L_0^*w=0$ in the sense of distributions. By the kernel characterization in Theorem~\ref{periodictheorem}\,\textup{(ii)}, $w$ is a constant multiple of $\phi_0$. Since $\av{w}=\av{\phi_0}=1$, we obtain $w=\phi_0$.

\smallskip\noindent
\emph{Step 3. Convergence of the full sequence.}
Steps~1 and~2 show that every subsequence has a further subsequence converging to $\phi_0$ in $L^2(\T)$. Hence $w_N\to\phi_0$ in $L^2(\T)$. Using $\phi_N=g_Nw_N$, we conclude that
\begin{equation*}
 \|\phi_N-g_N\phi_0\|_2
 \le\|g\|_\infty\|w_N-\phi_0\|_2\longrightarrow0,
\end{equation*}
which proves \eqref{strongapprox}.
\end{proof}

We now prove Proposition~\ref{onestep}. To this aim, we first use \eqref{smalljumpestimate} to choose $r$, and then apply Lemma~\ref{highfreq} with this radius $r$ fixed to choose $N$.

\begin{proof}[Proof of Proposition~\ref{onestep}]
For $r>0$, define
\begin{equation}\label{resetcoefficient}
 \bar a_r(x,y)=\chi_r(y)a_*+(1-\chi_r(y))a(x,y).
\end{equation}
By \eqref{gchoice}, \eqref{newcoefficient}, and \eqref{resetcoefficient}, both $\bar a_r$ and $a_{r,N}$ satisfy \eqref{ellipticity}. They are smooth for $y\ne0$, $\ell$-periodic in $x$, and radial in $y$. They equal $a_*$ and $a_*/g_N$, respectively, for $|y|\le r/2$, and equal $a_*$ for sufficiently large $|y|$. Thus both satisfy Definition~\ref{finiteclass}. Denote by $\bar\phi_r$ and $\phi_{r,N}$ the normalized positive adjoint solutions  for the operators related to $\bar a_r$ and $a_{r,N}$, respectively(see Definition~\ref{adjsol}).

\smallskip\noindent
\emph{Step 1. An upper bound for $J(\bar\phi_r)$ and the choice of $r$.}
Define
\begin{equation}\label{concentrationerror}
 \varepsilon=\frac{\tau-\rho_g}{1+\rho_g}J(\phi)>0.
\end{equation}
For every positive smooth $\widetilde\phi$, the concavity of the square root gives
\begin{equation}\label{Jtangent}
 J(\widetilde\phi)
 \le J(\phi)+\frac12\av{\phi^{-1/2}(\widetilde\phi-\phi)}.
\end{equation}
To estimate the last term with $\widetilde\phi=\bar\phi_r$, first observe that
$
 \av{(\phi^{-1/2}-J(\phi))\phi}
 =J(\phi)-J(\phi)\av{\phi}=0$.
Theorem~\ref{periodicsolvability}\,\textup{(iii)} therefore gives a unique smooth solution of
\begin{equation}\label{Jpoisson}
 L_a\psi=\phi^{-1/2}-J(\phi),\qquad \av{\psi\phi}=0.
\end{equation}
This function $\psi$ depends on $a$ and $\phi$, but is independent of $r$ and $N$.

Using \eqref{Jpoisson}, the normalizations $\av{\bar\phi_r}=\av{\phi}=1$, and $L_{\bar a_r}^*\bar\phi_r=0$, we obtain
\begin{equation}\label{Jduality}
 \av{\phi^{-1/2}(\bar\phi_r-\phi)}=\av{(\phi^{-1/2}-J(\phi))\bar\phi_r}=\av{L_a\psi\,\bar\phi_r}=\av{(L_a-L_{\bar a_r})\psi\,\bar\phi_r}.
\end{equation}
By \eqref{resetcoefficient}, $a=\bar a_r$ for $|y|\ge r$. Since $\bar\phi_r>0$ and $\av{\bar\phi_r}=1$, \eqref{Jduality} and \eqref{smalljumpestimate} give
\begin{equation*}
 \left|\av{\phi^{-1/2}(\bar\phi_r-\phi)}\right|
 \le\|(L_a-L_{\bar a_r})\psi\|_\infty
 \le C_*r^{2-\sigma}\|D^2\psi\|_\infty,
\end{equation*}
where $C_*$ is given by Lemma \ref{smalljump}.
Substituting this estimate into \eqref{Jtangent} yields
\begin{equation}\label{resetJ}
 J(\bar\phi_r)
 \le J(\phi)+\frac{C_*}{2}r^{2-\sigma}\|D^2\psi\|_\infty.
\end{equation}
Since $2-\sigma>0$ and $\psi$ is independent of $r$, choose
\begin{equation*}
 0<r<\min\left\{R_0,
 \left(\frac{2\varepsilon}
 {C_*(1+\|D^2\psi\|_\infty)}\right)^{1/(2-\sigma)}\right\}.
\end{equation*}
By \eqref{resetJ}, this choice gives
\begin{equation}\label{resetJsmall}
 J(\bar\phi_r)\le J(\phi)+\varepsilon.
\end{equation}

\smallskip\noindent
\emph{Step 2. The high-frequency limit for the selected $r$.}
Keep $r$ fixed. Let $q_r$ be given by \eqref{periodkernel} for $a$ and this selected $r$. By Lemma~\ref{periodization} below, $q_r$ satisfies \eqref{qassumption}. Denote by $Q_r$ and $b_r$ the corresponding operator and function defined in \eqref{Qdefinition}. Then \eqref{newcoefficient}, \eqref{resetcoefficient}, and \eqref{tailrepresentation} give
\begin{equation}\label{perturbedoperators}
 L_{\bar a_r}=A_r+Q_r-b_r,\qquad
 L_{a_{r,N}}=g_N^{-1}A_r+Q_r-b_r.
\end{equation}
Since $q_r$ is independent of $N$, Lemma~\ref{highfreq}  with $\phi_0=\bar\phi_r$ applies to the operators in \eqref{perturbedoperators} and gives 
\begin{equation}\label{perturbedadjointlimit}
 \lim_{N\to\infty}\|\phi_{r,N}-g_N\bar\phi_r\|_2=0.
\end{equation}
Using $|\sqrt{s}-\sqrt{t}|\le\sqrt{|s-t|}$ for $s,t\ge0$ and the normalized measure on $\T$, we obtain
\begin{equation*}
 \left|J(\phi_{r,N})-\av{\sqrt{g_N\bar\phi_r}}\right|\le\av{|\phi_{r,N}-g_N\bar\phi_r|^{1/2}}\le\|\phi_{r,N}-g_N\bar\phi_r\|_1^{1/2}\le\|\phi_{r,N}-g_N\bar\phi_r\|_2^{1/2}
 \rightarrow0,
\end{equation*}
as $N\rightarrow\infty$. Applying \eqref{averagelimit} with $G=\sqrt{g}$ and $\Phi=\sqrt{\bar\phi_r}$ gives
$
 \lim_{N\to\infty}\av{\sqrt{g_N\bar\phi_r}}
 =\rho_gJ(\bar\phi_r)$.
Combining the last two relations, we conclude that
\begin{equation}\label{Jhighfreqlimit}
 \lim_{N\to\infty}J(\phi_{r,N})=\rho_gJ(\bar\phi_r).
\end{equation}

\smallskip\noindent
\emph{Step 3. Selection of $N$ and the contraction estimate.}
By \eqref{Jhighfreqlimit}, choose an integer $N>N_0$ such that
$
 J(\phi_{r,N})\le\rho_gJ(\bar\phi_r)+\varepsilon$.
Combining this inequality with \eqref{resetJsmall} and \eqref{concentrationerror}, we obtain
\begin{equation*}
 J(\phi_{r,N})
 \le\rho_gJ(\phi)+(1+\rho_g)\varepsilon
 =\tau J(\phi).
\end{equation*}
This proves \eqref{onestepcontraction} and completes the proof.
\end{proof}

\section{Proofs of the main results}\label{sec:fixed}

We iterate the contraction estimate from Proposition~\ref{onestep} to prove Theorem~\ref{main}.

\begin{proof}[Proof of Theorem~\ref{main}]
\smallskip\noindent
\emph{Step 1. Auxiliary coefficients and Poisson solutions.}
Fix $a_*$ as in Section 2, $g$ and $\rho_g$ as in \eqref{gchoice}-\eqref{rhochoice},  $\tau\in(\rho_g,1)$, the functions and constants $\theta,D,d,\ell,H,\zeta$ from Section~\ref{sec:poisson}, and the cutoff $\eta$ used in \eqref{extractiondata}. These choices are independent of $j$. Set $a_0=a_*$ and $\phi_0=1$. Applying Proposition~\ref{onestep} with $R_0=1$ and $N_0=0$ gives $r_1,N_1,a_1,\phi_1$ such that $J(\phi_1)\le\tau$.

Suppose $a_j,\phi_j,r_j,N_j$ have been constructed. Define
\begin{equation}\label{hjdefinition}
 h_j=\eta(\phi_j),\qquad F_j=\{\phi_j>1/2\},\qquad c_j=\av{h_j\phi_j}.
\end{equation}
Theorem~\ref{periodicsolvability} and the definition of $c_j$ give a unique smooth periodic solution of
\begin{equation*}
 L_{a_j}v_j=c_j-h_j,\qquad\min_\T v_j=0.
\end{equation*}
Proposition~\ref{localization}\,\textup{(i)} gives $0\le v_j\le d$. Define $z_j=v_j+2d-4\theta d$. Since $z_j$ has bounded second derivatives, the number
\begin{equation}\label{Rbudget}
 R_{j+1}=\min\left\{\frac{r_j}{4},
 \left[\frac1{4C_*(1+\|D^2z_j\|_\infty)}\right]^{1/(2-\sigma)}\right\}
\end{equation}
is positive. Apply Proposition~\ref{onestep} to $a=a_j$ and $\phi=\phi_j$, with radius bound $R_{j+1}$ and frequency lower bound $N_j$. It gives $0<r_{j+1}<R_{j+1}$, $N_{j+1}>N_j$, and
\begin{equation}\label{recursion}
 a_{j+1}(x,y)=\chi_{r_{j+1}}(y)\frac{a_*}{g(N_{j+1}x_1/\ell)}
 +(1-\chi_{r_{j+1}}(y))a_j(x,y),
\end{equation}
with $J(\phi_{j+1})\le\tau J(\phi_j)$. In this construction, $z_j$ and hence $R_{j+1}$ are determined before $r_{j+1}$ and $N_{j+1}$ are chosen. The induction is therefore well-defined for every $j\ge1$. Induction gives $J(\phi_j)\le\tau^j$. Lemma~\ref{extraction} therefore yields
\begin{equation}\label{Fdecay}
 \frac{|F_j|}{\ell^n}\le\sqrt2\tau^j,
\end{equation}
and
\begin{equation}\label{cjbound}
 c_j\ge1-\sqrt2\tau^j.
\end{equation}
The choice \eqref{Rbudget} also ensures that
\begin{equation}\label{protected}
 C_*r_{j+1}^{2-\sigma}\|D^2z_j\|_\infty
 \le\frac{\|D^2z_j\|_\infty}{4(1+\|D^2z_j\|_\infty)}\le\frac14.
\end{equation}

\smallskip\noindent
\emph{Step 2. The limiting coefficient.}
Since $r_{j+1}<r_j/4$, we have $r_j\le4^{1-j}r_1$ and $\lim_{j\to\infty}r_j=0$. Since $\chi_{r_{j+1}}(y)\frac{a_*}{g(N_{j+1}x_1/\ell)}$ in \eqref{recursion} is supported in $|y|<r_{j+1}$,
\begin{equation*}
 a_k(x,y)=a_j(x,y)\quad\text{for }k\ge j\text{ and }|y|\ge r_{j+1}.
\end{equation*}
Consequently, for every $y\ne0$ the coefficient is unchanged after finitely many stages. Define
\begin{equation*}
 a_\infty(x,y)=\lim_{j\to\infty}a_j(x,y).
\end{equation*}
Let $K\subset\R^n\times(\R^n\setminus\{0\})$ be compact, and set $\delta_K=\min_K|y|>0$. Choose $m_K$ with $r_{m_K+1}<\delta_K/2$. Then $a_\infty=a_{m_K}$ on $\{|y|>\delta_K/2\}$, a neighborhood of $K$. This proves that $a_\infty$ is smooth for $y\ne0$. The bounds \eqref{ellipticity}, periodicity in $x$, and radiality in $y$ also hold for $a_\infty$. Moreover, $a_\infty=a_*$ for $|y|\ge r_1$.

\smallskip\noindent
\emph{Step 3. Perturbation estimates for the functions $z_j=v_j+2d-4d\theta$.}
Fix $j$. The construction gives $a_\infty=a_j$ whenever $|y|\ge r_{j+1}$. Since $z_j\in C_b^2(\R^n)$, Lemma~\ref{smalljump} and \eqref{protected} imply
\begin{equation*}
 \|(L_{a_\infty}-L_{a_j})z_j\|_\infty
 \le C_*r_{j+1}^{2-\sigma}\|D^2z_j\|_\infty\le\frac14.
\end{equation*}
Thus \eqref{errorbudget} holds for every $z_j$ with $\widetilde a=a_\infty$. The estimate depends only on $z_j$ and the radius $r_{j+1}$ selected after solving its Poisson equation.

\smallskip\noindent
\emph{Step 4. Localization and vanishing support measure.}
Choose $j_0$ such that $\sqrt2\tau^j\le1/4$ for $j\ge j_0$. Equation~\eqref{cjbound} then gives $c_j\ge3/4$ for $j\ge j_0$. The estimate from Step~3 allows us to apply Proposition~\ref{localization}\,\textup{(ii)} with $a=a_j$ and $\widetilde a=a_\infty$. Hence the functions $u_j=H(z_j)$ and $f_j=\zeta h_j$
belong to $C_c^\infty(B_1)$, have $u_j\ge0$ and $0\le f_j\le1$, and satisfy
\begin{equation*}
 L_{a_\infty}u_j\ge-f_j\quad\text{in }B_1,\qquad u_j\ge d\quad\text{on }B_{1/4}.
\end{equation*}
Let $F_j^{\mathrm{per}}$ denote the periodic extension of $F_j$ in \eqref{hjdefinition}, and set $E_j=B_1\cap F_j^{\mathrm{per}}$. This set is open. Since $\supp\zeta\subset B_1$ and  \eqref{hjdefinition} gives $\supp h_j\subset\{\phi_j\ge1\}$,
\begin{equation*}
 \supp f_j\subset\supp\zeta\cap\{\phi_j\ge1\}
 \subset B_1\cap F_j^{\mathrm{per}}=E_j.
\end{equation*}
The ball $B_1$ intersects at most $M_{\ell,n}$ period cells, where $M_{\ell,n}<\infty$ is independent of $j$. Equation~\eqref{Fdecay} gives
\begin{equation*}
 |E_j|\le M_{\ell,n}|F_j|
 \le M_{\ell,n}\ell^n\sqrt2\tau^j.
\end{equation*}
It follows that $|E_j|\to0$ as $j\to\infty$. 
After omitting the terms with $j<j_0$ and reindexing
the sequences, we obtain all the conclusions of
Theorem~\ref{main} with $a=a_\infty$.
\end{proof}

We next prove the two corollaries. The first follows directly from the vanishing support measure; the second uses Dirichlet solvability and comparison with the smooth subsolutions.

\begin{proof}[Proof of Corollary~\ref{noestimate}]
Since $0\le f_j\le\mathbf1_{E_j}$, one has $\|f_j\|_{L^p(B_1)}\le|E_j|^{1/p}$. If the estimate held, then
\begin{equation*}
 0<d\le\sup_{B_1}u_j\le C_a|E_j|^{1/p},
\end{equation*}
which contradicts $|E_j|\to0$. 
\end{proof}

\begin{proof}[Proof of Corollary~\ref{equationcounterexample}]
Fix $j$.  Since $a$ is smooth for $y\ne0$ and satisfies
\eqref{ellipticity}, applying Theorem~5.6 in Mou~\cite{Mou17} with
$I(x,t,\varphi)=-L_a\varphi(x)-f_j(x)$ and zero exterior
data, we get a bounded continuous viscosity solution
$w_j$ of \eqref{dirichletcounterexample}. We prove that $w_j\ge u_j$ in $\R^n$.
Suppose instead that
$
 M:=\max_{\R^n}(u_j-w_j)>0$.
Since $w_j$ and $u_j$ vanish outside $B_1$,   $M=u_j(x_0)-w_j(x_0)$ for some $x_0\in B_1$.
Choose $\xi\in C_c^\infty(\R^n)$ with $0\le\xi\le1$
and $\xi=1$ on $B_1$, and define
$
 \varphi=u_j-M\xi$.
In $B_1$, we have $\varphi=u_j-M\le w_j$, with equality
at $x_0$. Outside $B_1$, we have
$\varphi=-M\xi\le0=w_j$.
Thus $\varphi$ touches $w_j$ globally from below at $x_0$.
Moreover, evenness and \eqref{ellipticity} give
\begin{equation*}
 L_a\xi(x_0)
 =2\int_{\R^n}
 \bigl(\xi(x_0+y)-1\bigr)
 \frac{a(x_0,y)}{|y|^{n+\sigma}}\dd y<0,
\end{equation*}
because $\xi\le1$ and $\xi$ has compact support. Let $d>0$ and $u_j$ be given by Theorem ~\ref{main},
the viscosity supersolution inequality and
$L_au_j\ge-f_j$ therefore yield a contradiction
$$-f_j(x_0)\ge L_a\varphi(x_0)
 =L_au_j(x_0)-M L_a\xi(x_0)\ge-f_j(x_0)-M L_a\xi(x_0)
 >-f_j(x_0);$$
hence $w_j\ge u_j\ge0$ in $\R^n$, and
$w_j\ge d$ on $B_{1/4}$. Finally, $\widetilde w_j=-w_j$ satisfies
$L_a\widetilde w_j=f_j$ in $B_1$ in the viscosity sense
and $\widetilde w_j=0$ in $\R^n\setminus B_1$.
Since $\{f_j>0\}\subset E_j$, we have
$
 \lim_{j\to\infty}|\{f_j>0\}|=0$.
The lower bound for $w_j$ gives
$\|\widetilde w_j\|_{L^\infty(B_1)}\ge d$ for every $j$.
\end{proof}

\begin{remark}\label{radialdependence}
The limiting coefficient depends nontrivially on $|y|$. Since $r_2<r_1/4$, every perturbation after stage one vanishes on the annulus $r_2\le|y|\le r_1/2$. Thus
\begin{equation*}
 a_\infty(x,y)=\frac{a_*}{g(N_1x_1/\ell)}\quad\text{for }r_2\le|y|\le r_1/2,
\end{equation*}
whereas $a_\infty(x,y)=a_*$ for $|y|\ge r_1$. Choose $x_0$ with $g(N_1x_{0,1}/\ell)\ne1$, a nonzero $y_0$ in the annulus, and $t_0>0$ with $|t_0y_0|\ge r_1$. Let $x_{0,1}$ be the first component of $x_0$.  Then
\begin{equation*}
 a_\infty(x_0,t_0y_0)=a_*\ne\frac{a_*}{g(N_1x_{0,1}/\ell)}=a_\infty(x_0,y_0).
\end{equation*}
Thus $a_\infty$ is not homogeneous of degree zero in $y$, whereas every coefficient in \eqref{quadratic-kernel} is homogeneous of degree zero in $y$.
\end{remark}
\appendix
\section{Periodic solvability}\label{sec:linear}

\begin{lemma}\label{fourier}
For every $r>0$, let $A_r$ be given by \eqref{Ardefinition}.  Then $1-A_r:H^{s+\sigma}(\T)\to H^s(\T)$ is an isomorphism and \eqref{ellest} holds for $s\in\R$. The operator $A_r$ is self-adjoint and nonpositive in $L^2(\T)$ with domain $H^\sigma(\T)$.
\end{lemma}
\begin{proof}  With  the Fourier and Sobolev notation in \eqref{fouriernotation}-\eqref{sobolevnotation}, it follows from the definitions of $A_r, e_k$ in \eqref{Ardefinition}  and \eqref{fouriernotation} that 
 $A_re_k=-\mu_r(k)e_k$, where
\begin{equation}\label{murdefinition}
 \mu_r(k)=2a_*\int_{\R^n}(1-\cos(2\pi k\cdot y/\ell))
 \frac{\chi_r(y)}{|y|^{n+\sigma}}\dd y\ge0.
\end{equation}
For $k\ne0$, choose an orthogonal matrix $O$ such that
$Oe_1=k/|k|$ and make the change of variables
$y=\ell Oz/(2\pi|k|)$. Then
\begin{equation*}
 2a_*\int_{\R^n}
 \frac{1-\cos(2\pi k\cdot y/\ell)}{|y|^{n+\sigma}}\dd y=2a_*\left(\frac{2\pi}{\ell}\right)^\sigma
 |k|^\sigma
 \int_{\R^n}\frac{1-\cos z_1}{|z|^{n+\sigma}}\dd z=\kappa|k|^\sigma,
\end{equation*}
where
\begin{equation*}
 \kappa:=2a_*\left(\frac{2\pi}{\ell}\right)^\sigma
 \int_{\R^n}\frac{1-\cos z_1}{|z|^{n+\sigma}}\dd z
 \in(0,\infty).
\end{equation*}
Since $1-\chi_r$ vanishes on $B_{r/2}$,
\begin{equation*}
 \begin{aligned}
 	0\le\kappa|k|^\sigma-\mu_r(k)
 	&=2a_*\int_{\R^n}
 	\frac{(1-\cos(2\pi k\cdot y/\ell))(1-\chi_r(y))}{|y|^{n+\sigma}}\dd y\\
 	&\le4a_*\int_{|y|\ge r/2}|y|^{-n-\sigma}\dd y
 	=\frac{4a_*|\mathbb S^{n-1}|}{\sigma}(r/2)^{-\sigma}=:C r^{-\sigma}.
 \end{aligned}
\end{equation*}
Thus, there exists $K_0>0$ such that $\mu_r(k)\ge\kappa|k|^\sigma/2$ for all $|k|>K_0$.
For the remaining finitely many $k$ with $|k|\leq K_0$, $1+\mu_r(k)\ge1$. Together with $\mu_r(k)\le\kappa|k|^\sigma$, we get
\begin{equation}\label{symbol}
	c_r(1+|k|^2)^{\sigma/2}\le1+\mu_r(k)
	\le C_r(1+|k|^2)^{\sigma/2},\qquad k\in\Z^n.
\end{equation}
Fix $s\in\R$. By \eqref{sobolevnotation} and
\eqref{symbol},
\begin{equation}\label{updowm}
 c_r\|v\|_{H^{s+\sigma}}
 \le\|(1-A_r)v\|_{H^s}
 \le C_r\|v\|_{H^{s+\sigma}},
 \qquad v\in H^{s+\sigma}(\T).
\end{equation}
Thus $1-A_r:H^{s+\sigma}(\T)\to H^s(\T)$ is bounded
and injective. For any $f\in H^s(\T)$, define $v$ by
$\widehat v(k)=\widehat f(k)/(1+\mu_r(k))$.
The lower bound in \eqref{symbol} gives
\begin{equation*}
 \|v\|_{H^{s+\sigma}}^2
 =\sum_{k\in\Z^n}
 \frac{(1+|k|^2)^{s+\sigma}|\widehat f(k)|^2}
      {(1+\mu_r(k))^2}
 \le c_r^{-2}\|f\|_{H^s}^2.
\end{equation*}
Hence $v\in H^{s+\sigma}(\T)$ and $(1-A_r)v=f$.
This proves that $1-A_r$ is an isomorphism.

If $v$ is a distribution with
$v,A_rv\in H^s(\T)$, then $f=v-A_rv\in  H^s(\T)$, and a similar Fourier analysis shows that
$v\in H^{s+\sigma}(\T)$. Therefore,
\eqref{updowm} and the triangle inequality yield
\begin{equation*}
 \|v\|_{H^{s+\sigma}}
 \le c_r^{-1}\|(1-A_r)v\|_{H^s}
 \le c_r^{-1}
 \bigl(\|v\|_{H^s}+\|A_rv\|_{H^s}\bigr),
\end{equation*}
which proves \eqref{ellest}.

 Next, regard $A_r$ as an operator in
 $L^2(\mathbb T)$ with domain
 $D(A_r)=H^\sigma(\mathbb T)$.
 This domain is dense in $L^2(\mathbb T)$.
 Parseval's identity gives
 \[
 (A_ru,v)_2
 =
 -\sum_{k\in\mathbb Z^n}
 \mu_r(k)\widehat u(k)\overline{\widehat v(k)}
 =
 (u,A_rv)_2,
 \qquad u,v\in H^\sigma(\mathbb T).
 \]
 Thus $A_r$ is symmetric.
 
 Conversely, let $v\in D(A_r^*)$ and set $h=A_r^*v$.
 By definition,
$
 (A_ru,v)_2=(u,h)_2$
for every $u\in H^\sigma(\mathbb T)$.
 Testing with $u=\cos(2\pi k\cdot x/\ell)$ and $u=\sin(2\pi k\cdot x/\ell)$, and combing the resulting identities, yields
 \[
 \widehat h(k)=-\mu_r(k)\widehat v(k),
 \qquad k\in\mathbb Z^n.
 \]
 Hence $A_rv=h$ in the sense of distributions.
 Since $v,h\in L^2(\mathbb T)$, the regularity statement
 proved above gives $v\in H^\sigma(\mathbb T)$.
 It follows that $D(A_r^*)=D(A_r)$ and $A_r^*=A_r$.
 
 Finally, for every $v\in H^\sigma(\mathbb T)$,
$
 (A_rv,v)_2
 =
 -\sum_{k\in\mathbb Z^n}
 \mu_r(k)|\widehat v(k)|^2
 \le0$. 
 Thus $A_r$ is self-adjoint and nonpositive in
 $L^2(\mathbb T)$ with domain $H^\sigma(\mathbb T)$.
\end{proof}

\Needspace{6\baselineskip}
\begin{lemma}\label{periodization}
Let $a$ be an auxiliary coefficient as in
Definition~\ref{finiteclass}, and fix $r>0$.
Define $q$ by \eqref{periodkernel} and $Q,b$ by
\eqref{Qdefinition}. Then $q$ satisfies
\eqref{qassumption}, and \eqref{tailrepresentation}
holds for every $v\in C^\infty(\T)$.
If, in addition, $r\le\rho$, then
\begin{equation*}
 L_av=cA_rv+Qv-bv,
 \qquad v\in C^\infty(\T),
\end{equation*}
where $\rho$ and $c$ are as in
Definition~\ref{finiteclass}.
\end{lemma}

\begin{proof}
We first verify \eqref{qassumption}.
Since $1-\chi_r$ vanishes near zero, each term of the series 
in \eqref{periodkernel}, extended by zero at
$z-x+\ell k=0$, is smooth.
For $x,z\in Q_\ell$ and all sufficiently large $|k|$,
one has $|z-x+\ell k|>\max\{r,R\}$, where $R$ given by Definition~\ref{finiteclass};  therefore, 
$$ \frac{(1-\chi_r(z-x+\ell k))a(x,z-x+\ell k)}
 {|z-x+\ell k|^{n+\sigma}}=a_*|z-x+\ell k|^{-n-\sigma}.$$  For every pair of multi-indices $\alpha,\beta$,
\begin{equation*}
 \left|
 \partial_x^\alpha\partial_z^\beta
 \bigl(a_*|z-x+\ell k|^{-n-\sigma}\bigr)
 \right|
 \le C_{\alpha,\beta}
 (1+|k|)^{-n-\sigma-|\alpha|-|\beta|}
\end{equation*}
Therefore, each term of the series 
in \eqref{periodkernel} and all its differentiated series
converge uniformly on $Q_\ell\times Q_\ell$.
Reindexing the sum and using the periodicity of $a$
in $x$ give
\begin{equation*}
 q(x+\ell m,z)=q(x,z+\ell m)=q(x,z),
 \qquad m\in\Z^n.
\end{equation*}
It follows that $q\in C^\infty(\T\times\T)$.
 By the analysis above, we have  $q(x,z)>0$ for each $(x,z)$. Hence
$
 0<q_-:=\min_{\T\times\T}q
 \le q(x,z)
 \le\max_{\T\times\T}q=:q_+<\infty$.
This proves \eqref{qassumption}.

Next, fix $v\in C^\infty(\T)$.
By evenness in $y$, followed by the change of variables
$y=z-x+\ell k$ on each cell $Q_\ell-x+\ell k$,
\begin{equation*}
 \begin{aligned}
 &\int_{\R^n}\delta_yv(x)
   \frac{(1-\chi_r(y))a(x,y)}{|y|^{n+\sigma}}\dd y=2\int_{\R^n}(v(x+y)-v(x))
   \frac{(1-\chi_r(y))a(x,y)}{|y|^{n+\sigma}}\dd y\\
 &\quad=2\sum_{k\in\Z^n}\int_{Q_\ell}
   (v(z)-v(x))
   \frac{(1-\chi_r(z-x+\ell k))a(x,z-x+\ell k)}
        {|z-x+\ell k|^{n+\sigma}}\dd z\\
 &\quad=\int_\T(v(z)-v(x))q(x,z)\,d\bar z=Qv(x)-b(x)v(x).
 \end{aligned}
\end{equation*}
Here we used $v(z+\ell k)=v(z)$ and the definitions
\eqref{periodkernel} and \eqref{Qdefinition}.
This proves \eqref{tailrepresentation}.

Finally, if $r\le\rho$, then
$a(x,y)=a_*c(x)$ for $0<|y|<r$.
Since $\supp\chi_r\subset B_r$, \eqref{Ardefinition}
gives
\begin{equation*}
 \int_{\R^n}\delta_yv(x)
 \frac{\chi_r(y)a(x,y)}{|y|^{n+\sigma}}\dd y
 =c(x)A_rv(x).
\end{equation*}
Adding this identity to \eqref{tailrepresentation}
and using \eqref{operator} gives
$L_av=cA_rv+Qv-bv$.
\end{proof}
\begin{proof}[Proof of Theorem~\ref{periodicsolvability}]
\textup{(i)}
By Lemma~\ref{fourier} and the positivity of $c$, the operator
$c(A_r-1):H^\sigma(\T)\to L^2(\T)$ is an isomorphism.
The operator $Q+c-b$ is bounded on $L^2(\T)$, and the
embedding $H^\sigma(\T)\hookrightarrow L^2(\T)$ is compact.
Thus
$
 L=c(A_r-1)+(Q+c-b)
$
is a compact perturbation of an isomorphism, so $L$ is
Fredholm of index zero
\cite{Brezis11}.

We first establish the regularity of solutions.
Suppose that $v\in H^\sigma(\T)$ satisfies $Lv=f$ with
$f\in C^\infty(\T)$. Then
\begin{equation*}
 (1-A_r)v=(1-b/c)v+c^{-1}Qv-c^{-1}f.
\end{equation*}
The smoothness of $q$ gives
$\|Qv\|_{H^s}\le C_s\|v\|_2$ for $s\ge0$.
Since $b$ and $c$ are smooth and $c>0$, $(1-b/c)v+c^{-1}Qv-c^{-1}f$ belongs to $H^s(\T)$ whenever $v\in H^s(\T)$.
Lemma~\ref{fourier} therefore yields
$
 \|v\|_{H^{s+\sigma}}
 \le C_s\bigl(\|v\|_{H^s}+\|f\|_{H^s}\bigr)$ for 
$ s\ge0$,
where we have used $\|v\|_2\le\|v\|_{H^s}$.
Starting from $v\in H^\sigma(\T)$ and iterating with
$s=\sigma,2\sigma,\ldots$, we obtain
$v\in H^{j\sigma}(\T)$ for every integer $j\ge1$.
Sobolev embedding gives $v\in C^\infty(\T)$.

For  $v\in\ker L$,
denote $x_0$ a maximum point of 
$v$. We have
\begin{equation*}
 0=Lv(x_0)
 \le\int_\T\bigl(v(z)-v(x_0)\bigr)q(x_0,z)\,d\bar z
 \le q_-\bigl(\av{v}-v(x_0)\bigr)
 \le0.
\end{equation*}
Hence $\av{v(x_0)-v}=0$, and $v$ is constant. Conversely, $L1=0$.
Thus $\ker L=\operatorname{span}\{1\}$.

\medskip
\noindent\textup{(ii)}
Since $L$ is Fredholm of index zero and
$\ker L=\operatorname{span}\{1\}$, its range is closed
and has codimension one.  Choose a nonzero function $\phi$
spanning $(\operatorname{Ran}L)^\perp$. Then
$\av{\phi L\varphi}=0$ for every smooth periodic $\varphi$,
so $L^*\phi=0$.

Conversely, suppose that $\Psi\in L^2(\T)$ satisfies
$L^*\Psi=0$. For any $v\in H^\sigma(\T)$,
choose $\varphi_k\in C^\infty(\T)$ such that
$\lim_{k\to\infty}\|\varphi_k-v\|_{H^\sigma}=0$.
The boundedness of $L:H^\sigma(\T)\to L^2(\T)$ gives
$
 |\av{\Psi Lv}|
 =|\av{\Psi L(v-\varphi_k)}|
 \le C\|\Psi\|_2\|v-\varphi_k\|_{H^\sigma}.
$
Letting $k\to\infty$, we obtain
$\Psi\in(\operatorname{Ran}L)^\perp$.
Thus the adjoint nullspace in $L^2(\T)$
is $\operatorname{span}\{\phi\}$, and
\begin{equation}\label{fredholmrange}
 \operatorname{Ran}L
 =\bigl\{f\in L^2(\T):\av{f\phi}=0\bigr\}.
\end{equation}

We next prove that $\phi$ is smooth.
By the self-adjointness of $A_r$,
\begin{equation}\label{adjointformula}
 L^*\phi=A_r(c\phi)+Q^*\phi-b\phi,
 \qquad
 Q^*\phi(x)=\int_\T q(z,x)\phi(z)\,d\bar z.
\end{equation}
Define $w=c\phi\in L^2(\T)$. Then
\begin{equation}\label{adjointbootstrap}
 (1-A_r)w=Q^*\phi+(1-b/c)w.
\end{equation}
The smoothness of $q$ gives
$
 \|Q^*\phi\|_{H^s}\le C_s\|\phi\|_2$ for 
$s\ge0$.
If $w\in H^s(\T)$, Lemma~\ref{fourier} applied to
\eqref{adjointbootstrap} yields
$
 \|w\|_{H^{s+\sigma}}
 \le C_s\bigl(\|w\|_{H^s}+\|\phi\|_2\bigr)$. Via
iteration and Sobolev
embedding, we get $w\in C^\infty(\T)$.
Since $c$ is smooth and positive,
$\phi=w/c\in C^\infty(\T)$.

 We next prove positivity.
By \textup{(i)} and the maximum principle,
$\operatorname{Ran}L$ contains no smooth strictly
positive function. Indeed, a smooth solution of
$Lv=h>0$ would satisfy $h(x_0)=Lv(x_0)\le0$
at a maximum point $x_0$.
In particular, $1\notin\operatorname{Ran}L$, so
\eqref{fredholmrange} gives $\av{\phi}\ne0$.
We may therefore normalize $\phi$ by $\av{\phi}=1$.

If $\phi$ is negative somewhere, choose a nonzero
$\eta\in C^\infty(\T)$ with $\eta\ge0$ and
$\supp\eta\subset\{\phi<0\}$. Then
$\av{\eta\phi}<0$, and the smooth function
$
 h=1-{\eta}/{\av{\eta\phi}}\ge1
$
satisfies $\av{h\phi}=0$.
By \eqref{fredholmrange}, $h\in\operatorname{Ran}L$.
Since $h\ge1$, this contradicts the fact that
$\operatorname{Ran}L$ contains no smooth strictly
positive function.
Thus $\phi\ge0$.

If $\phi(x_0)=0$, then $w=c\phi$ attains its zero
minimum at $x_0$, and hence $A_rw(x_0)\ge0$.
By \eqref{adjointformula} and \eqref{qassumption},
$
 0=A_rw(x_0)+Q^*\phi(x_0)
 \ge q_-\av{\phi}=q_->0$, which is 
a contradiction. Therefore $\phi>0$.
Since the adjoint nullspace is one-dimensional,
the normalization $\av{\phi}=1$ determines $\phi$
uniquely.

\medskip
\noindent\textup{(iii)}
If $v$ is a smooth solution of $Lv=f$, duality gives
$
 \av{f\phi}
 =\av{\phi Lv}
 =\av{vL^*\phi}
 =0$.
Conversely, if $f\in C^\infty(\T)$ satisfies
$\av{f\phi}=0$, then \eqref{fredholmrange} gives
$v\in H^\sigma(\T)$ with $Lv=f$.
By \textup{(i)}, this solution is smooth. Since $\ker L=\operatorname{span}\{1\}$, the solution
is unique up to an additive constant.
Either $\av{v\phi}=0$ or $\min_\T v=0$ determines
that constant uniquely.
\end{proof}

\textbf{AI assistance statement. }The AI tool ChatGPT (OpenAI) was used in the preparation of this manuscript. The authors have reviewed all AI-assisted content and take full responsibility for the final manuscript.

\end{document}